\documentclass[12pt, reqno]{amsart}

\usepackage{amsmath,amsthm,amsfonts,amscd,amssymb}
\usepackage{times,euler,eucal,eufrak,graphicx}
\usepackage{hyperref}

\usepackage{caption} 
\usepackage{subcaption} 

\title[Estimates for compression norms]
{Estimates for compression norms and additivity violation in quantum information}
\author {Beno\^\i{}t Collins, Motohisa Fukuda, Ping Zhong}
\address{Department of Mathematics, University of Ottawa, Kyoto University} \email{collins@uottawa.ca}
\address{Zentrum Mathematik, M5, Technische Universit\"at M\"unchen, Boltzmannstrasse 3, 85748 Garching, Germany} 
\email{m.fukuda@tum.de}
\address{School of Mathematics and Statistics, Wuhan University, No.299, Ba Yi Road, Wuhan, Hubei 430072, China}
\email{pzhong@indiana.edu}

\theoremstyle{plain}
\newtheorem{lemma}{Lemma}[section]
\newtheorem{theorem}[lemma]{Theorem}
\newtheorem{proposition}[lemma]{Proposition}
\newtheorem{corollary}[lemma]{Corollary}

\theoremstyle{definition}
\newtheorem*{definition}{Definition}

\theoremstyle{remark}
\newtheorem*{remark}{Remark}
\newtheorem*{example}{Example}

 \DeclareMathOperator{\Tr}{Tr}

\newcommand{\C}{{\mathbb{C}}}
\newcommand{\R}{{\mathbb{R}}}

\begin{document}

\begin{abstract}
The free contraction norm (or the $(t)$-norm) was introduced in \cite{BCN1} as a tool to compute the typical
location of the collection of singular values associated to a random subspace of the tensor product of 
two Hilbert spaces. In turn, it was used in \cite{BCN2013} 
in order to obtain sharp bounds for the violation of the additivity of the minimum output entropy
for random quantum channels with Bell states. 
This free contraction norm, however, is difficult to compute explicitly. The purpose of this note is
to give a good estimate for this norm. Our technique is based on results of super convergence in the
context of free probability theory. As an application, we give a new, simple and conceptual
proof of the violation of the additivity of the minimum output entropy.
\end{abstract}

\maketitle

\section{Introduction}

In a $II_1$\ factor, let $M_k$ be the abelian $k$-dimensional subalgebra generated by
$k$ selfadjoint orthogonal projection of trace $1/k$ and $p_t$ be 
a projection of normalized trace $t\in (0,1]$ which is free from $M_k$.
Given a Hermitian operator $a\in M_k$, we define the quantity $||a||_{(t)}=||p_tap_t||$.
This notion was introduced in \cite{BCN1} and shown to be a norm. 
Its explicit computation is in general quite involved, even if not impossible.
In this paper, we are interested in bounds for $||a||_{(t)}$ in terms of quantities that are better understood
and easier to compute (combinations of $L^p$ norms in our case).
We are also interested in  applications of the bounds to quantum information theory.

\subsection{The problem and statement of results}

Given a Hermitian operator $a\in M_k$, we would
like to have simple  bounds for $||a||_{(t)}$. 
We are primarily interested in 
upper bounds, but our techniques also give lower bounds which have their own interests.

Given a projection $p_t$ free with $M_k$ of trace $t\in (0,1]$,
due to a result of Nica and Speicher \cite{NicaS1996}, the distribution of $t^{-1}p_tap_t$ is the same 
as the distribution of the fractional convolution power $\mu_a^{\boxplus 1/t}$
defined in the same paper, where $\mu_a$ is the spectral measure of $a$.   
Thus, the problem to estimate $||a||_{(t)}$ is reduced to study the bounds of the measure $\mu_a^{\boxplus 1/t}$.
When $1/t=n \in \mathbb{N}$, one can use Bercovici and Voiculescu's superconvergence result \cite{BV1995sup} to 
estimate the size of the support of $\mu_a^{\boxplus n}$. A strengthened inequality can be obtained
by using Kargin's result \cite{Kargin2007}, again for the case $1/t=n$. 

We adapt some techniques developed by Biane \cite{Biane1997} (see also \cite{Huang}) to study the support of $\mu_a^{\boxplus T}$, where $T=1/t$.
Our main results are Theorems  \ref{thm:upper-bound} and \ref{thm:lower-bound}. We summarize the
result in Theorems  \ref{thm:upper-bound} as follows and general results for upper and lower bounds will be given in Section 2.

\begin{theorem}
Given any Hermitian operator $a\in M_k$ such that $L^- \leq  a\leq L^+$ with $L^\pm \in \mathbb R$.
Let $\sigma(a)^2=\tau(a^2)-(\tau(a))^2$, and $t\in (0,1]$. 
Assume that for any atom $\xi_0$ of $\mu_a$, we have $\mu_a(\{ \xi_0\})\leq 1-1/k$, then we have
\begin{enumerate}
\item for any $0<t\leq 1/k$, the measure $\mu^{\boxplus 1/t}$ has no atomic part, and
\[
||a||_{(t)}\leq \max \left\{ \left| t L^\pm \pm 2 \sigma(a) \sqrt{t(1-t)} + (1-t) \tau(a) \right| \right\}.
\]
\item when $1\geq t>1/k$, 
\[
||a||_{(t)}\leq \max \left\{L^\pm, \left| t L^\pm \pm 2 \sigma(a) \sqrt{t(1-t)} + (1-t) \tau(a) \right| \right\}
\]
\end{enumerate}
\end{theorem}
This theorem is restated as Theorem \ref{thm:upper-bound},
where the statements are adapted to our applications.  

\subsection{Application to the additivity of the minimum output entropy}

In the seventies, Holevo introduced a one-shot (non-asymptotic) quantity $\chi$ \cite{Holevo1973},
which turned out to describe the rate of classical information transmission through a quantum channel
without entangled inputs \cite{Holevo1998,SchumacherWestmoreland1997}.
This quantity was conjectured to be additive, in the sense that for all quantum 
channels $\Phi, \Psi$,
\begin{equation}
\chi(\Phi \otimes \Psi) = \chi(\Phi) + \chi(\Psi).
\end{equation}
With an application of this property  the classical capacity of $\Phi$ would be equal to its one-shot capacity $\chi$ for any quantum channel $\Phi$.

Subsequently, Shor showed \cite{Shor2004} that the additivity of $\chi$ is equivalent to similar properties of other quantities of 
interest in quantum information, the foremost being the minimum output entropy of channels \cite{KingRuskai2001}
\begin{equation}
H^{\min}(\Phi) = \min_{X \in S_d} H(\Phi(X)).
\end{equation}
Here and below, $S_d$ stands for the set of quantum states:
\[S_d = \{ X\in M_d(\C) : X\geq 0, \mathrm{Tr}X=1\}.\]
The focus of the community shifted to showing additivity for the \emph{minimum output entropy (MOE)}, or its $p$-variants, 
called \emph{R\'enyi entropies}. These are defined for probability vectors $x \in \Delta_k$ by
\begin{equation}
H_p(x) = \frac{1}{1-p}\log\sum_{i=1}^k x_i^p,
\end{equation}
and extended by functional calculus to quantum states $X$. Note that the above definitions are valid for 
$p \in (0, \infty)$, the value in $p=1$, obtained by taking a limit, coinciding with the von Neumann entropy $H$. 
The $\min$ variants are defined by
\begin{equation}
H_p^{\min}(\Phi) = \min_{X \in S_d} H_p(\Phi(X)).
\end{equation}

The additivity property for the quantities $H_p^{\min}$ was shown to be false for $p>1$ in \cite{HaydenWinter2008} and 
later in \cite{AubrunSzarekWerner2010,CollinsNechita2011},
and concrete counterexamples were found for $p>4.79$ in \cite{WernerHolevo2002} and later for  $p>2$ in \cite{GHP2010}.
On the other hand,  additivity violation for $H^{\min}$ was proven by Hastings' \cite{Hastings2009} .
Later, generalizations and improvements were made 
in \cite{FKM2010,BrandaoHorodecki2010,FK2010,AubrunSzarekWerner2011,Fukuda2014,BCN2013}.
Since the resolution of the 
additivity conjecture, efforts have been made to understand more the additivity violation in terms of tensor-products of quantum channels 
in  \cite{CollinsNechita2010, CollinsNechita2010a, CollinsNechita2011,CollinsNechita2011a,CFN2012, CFN2013,FN2014}
into understanding, extending and improving the deviations from additivity.

As a main application of the norm estimate of this paper, we propose a new solution to this problem. 
A non-rigorous statement of additivity violation can be made as follows. 
\begin{theorem}\label{thm:violationF}
Let $\Phi$ be a random quantum channel defined on $M_n(\mathbb C)$,
with environment dimension $k$. 
For $1 \ll k \ll n$ with high probability,
$$ H^{\min} (\Phi \otimes \overline \Phi )  
<  H^{min} (\Phi) +H^{min} (\overline \Phi) .$$ 
\end{theorem}
A precise and more general statement of our result on violation is stated in Theorem \ref{thm:violation}.
We refer to Section \ref{sec:violation} to find the definition of random quantum channels 
and our proof methods.

\subsection{Organization of the paper}

This paper is organized as follows. 
First, in Section \ref{section:FP}, after recalling basic facts about freeness and the t-norm (Section \ref{sec:pre})
we provide the main estimates for the $(t)$-norm (Section \ref{sec:main}).
Next, in Section \ref{sec:violation} we turn to applications to QIT. We recall about the standard
strategy to prove existence of counterexamples of additivity of MOE (Section \ref{sec:strategy}) 
and set up our model (Section \ref{sec:model}). 
Then, our application is explained (Section \ref{sec:concentration}) to
prove the additivity violation (Section \ref{sec:violation2}).

\section*{Acknowledgement}

We would like to thank Ion Nechita and Serban Belinschi for useful conversations. 
BC was supported by NSERC, ERA, JSPS and AIMR.
MF was financially supported by the CHIST-ERA/BMBF project CQC and the John Templeton Foundation (ID$\sharp$48322).

\section{Norm estimates}\label{section:FP}

\subsection{Free probability: a primer and reminders}\label{sec:pre}
A $W^*$-probability space is a pair $(\mathcal{A},\tau)$, where
$\mathcal{A}$ is a finite von Neumann algebra and $\tau$ is 
a faithful, normal, trace state of $\mathcal{A}$. Elements in
$\mathcal{A}$ are called random variables. Let 
$\{\mathcal{A}_i\}_{i\in I}$ be a family of unital 
subalgebras of $\mathcal{A}$. They are called  \emph{free}
if for all $a_i\in\mathcal{A}_{j_i}$ such that $\tau(a_i)=0$, where $i=1,\cdots, n$,
we have
\[
 \tau(a_1\cdots a_n)=0,
\]
whenever we have $j_1\neq j_2, \cdots, j_{n-1}\neq j_n$. Collections of random variables
$\{S_i\}_{i\in J}$ are said to be free if the unital subalgebras generated by them are free.

Let $\{ a_i\}_{i=1}^k$ be a family of random variables and let 
$\mathbb{C} \langle X_1,\cdots,X_k\rangle$ be the algebra of noncommutative polynomials
on $\mathbb{C}$ generated by the $k$ free indeterminates $X_1,\cdots, X_k$.
The \emph{joint distribution} of the family $\{ a_i\}_{i=1}^k$ is the linear form
\[
\begin{split}
 \mu_{(a_1,\cdots,a_k)}:\mathbb{C}\langle X_1,\cdots,X_k \rangle&\rightarrow \mathbb{C}\\
        P&\mapsto \tau(P(a_1,\cdots,a_k)).
\end{split}
\]
Given two free random variables $a,b\in\mathcal{A}$, the distribution $\mu_{a+b}$
is uniquely determined by $\mu_a$ and $\mu_b$. The free additive convolution
of $\mu_a$ and $\mu_b$ is defined by $\mu_a\boxplus\mu_b=\mu_{a+b}$. 
When $x=x^*\in\mathcal{A}$, we identify $\mu_x$ with
the spectral measure of $x$ with respect to $\tau$.
The operation $\boxplus$ induces a binary operation 
on the set of probability measures on $\mathbb{R}$.

We now review analytic methods for calculating free convolutions. 
Denote by $M_{\mathbb{R}}$ the set of all probability measures
on $\mathbb{R}$. Given $\mu\in M_{\mathbb{R}}$, the Cauchy 
transform of $\mu$ is defined by 
\[
\mathit{G}_{\mu}(z)=\int_{\mathbb{R}}\frac{1}{z-t}d\mu(t), z\in\mathbb{C}^+.
\]
We set $\mathit{F}_{\mu}(z)=1/\mathit{G}_{\mu}(z)$. 
Given positive constants $\eta$ and $M$, we let
$\Gamma_{\eta,M}=\{z=x+iy:y>M,\text{and}\, |x|<\eta y \}$. For any $\eta>0$,
it is known from \cite{BV1993} that $\mathit{F}_{\mu}$ has a right inverse $\mathit{F}_{\mu}^{-1}$ defined
in $\Gamma_{\eta,M}$ for $M$ sufficiently large. The function $\varphi_{\mu}(z)=\mathit{F}_{\mu}^{-1}(z)-z$
is called the Voiculescu transform of $\mu$. Given two probability measures $\mu,\nu$ on $\mathbb{R}$,
the Voiculescu transform
linearizes free convolution. That is
    \begin{equation}
      \varphi_{\mu\boxplus \nu}(z)=\varphi_{\mu}(z)+\varphi_{\nu}(z)
    \end{equation}
for all $z$ in a region of the form $\Gamma_{\eta, M}$, where all three transforms are defined.    
We refer the reader to \cite{Basic} for an introduction to free convolutions.

\begin{definition}
For a positive integer k, embed $\mathbb{R}^k$ as a self-adjoint real subagebra $M_k$ of a $II_1$ factor $\mathcal{A}$,
spanned by $k$ mutually orthogonal projections of normalized trace $1/k$. Let $p_t$ be a projection of rank $t\in (0,1]$ in
$\mathcal{A}$, free from $M_k$. On the real vector space $\mathbb{R}^k$, we introduce the following quantity, called the $(t)$-norm:
\[
||a||_{(t)}:=||p_tap_t||_{\infty},
\]
where the vector $a\in\mathbb{R}^k$ is identified with its image in $M_k$.
\end{definition}
The reader can find some basic properties of $(t)$-norm from \cite{BCN1}.
This norm turned out to have a connection to the free convolution described above.
To see it, we recall the following result due to Nica and Speicher \cite{NicaS1996}.
\begin{theorem}\label{thm:nica-speicher}
Given a Hermitian operator $a\in M_k$, and $t\in (0,1]$, there exists
a probability measure $\mu_a^{\boxplus 1/t}$ such that
$\varphi_{\mu_a}^{\boxplus 1/t}(z)=(1/t) \varphi_{\mu_a}(z)$.
Moreover, we have
\[
\mu_{t^{-1}p_tap_t}=\mu_a^{\boxplus 1/t}.
\]
\end{theorem}
Therefore, $||a||_{(t)}$ is $t$ times the maximum between the upper bound and minus the lower bound of the support of the probability
measure $\mu_a^{\boxplus 1/t}$, and we try to get good bounds for them. 
For this purpose, we introduce the following notation; for a compactly supported measure $\mu$ on $\mathbb{R}$, we denote
\[
 [\mu]=\max\{|v|:v\in supp(\mu)\}.
\]
One can estimate $\mu^{\boxplus 1/t}$ for the case when $1/t$ is an integer, say $n$,
by Bercovici and Voiculescu's superconvergence result in \cite{BV1995sup}. 
The main result in this paper implies that
\[
 \lim_{n\rightarrow\infty}\frac{[\mu^{\boxplus n}]}{2\sqrt{n}}=1
\] 
when the first moment of $\mu$ is zero and the second moment of $\mu$ is one. 
Thus, we have
\[||a||_{(t)}= \tau(a)+2\sqrt{t}\sqrt{\tau(a^2)-\tau(a)^2}+o(\sqrt{t})\]
when $t=1/n$.
Kargin obtained a more precise bound for $[\mu^{\boxplus n}]$. By using \cite[Example $1$]{Kargin2007}, one can 
obtain the following inequality
\begin{equation}\label{eq:006}
||a||_{(t)}\leq  \tau(a)+2\sqrt{t}\sqrt{\tau(a^2)-\tau(a)^2}+
5t\frac{||a-\tau(a)||^3}{\tau(a^2)-\tau (a)^2}
\end{equation}
for $t=1/n$.

In \cite[Proposition 3.7]{BCN1}, another bound for $\mu^{\boxplus 1/t}$ was obtained. 
However, there is a minor mistake in the proof. 
We provide a slightly different approach based on techniques developed by Biane in \cite{Biane1997} and 
correct the mistake in \cite{BCN1}. 

\subsection{Main norm estimates}\label{sec:main}
In principle, from the Cauchy transform $G_\mu$ we can recover the probability measure $\mu$:
\begin{proposition} [Stieltjes-Peron inversion formula] \label{prop:SP}
\[
\mu [a,b] = - \lim_{\epsilon \to \infty} \frac 1 \pi \int_a^b \Im G_\mu (x - i \epsilon ) dx
\]
when $\mu(a)=\mu(b)=0$.
\end{proposition} 
In particular, if $G$ is continuous the non-atomic part of  $\mu$ is given by $-\frac 1 \pi \Im G_\mu(x)$.
Hence, for our sake, it is best to get the Cauchy transform $G_{\mu^{\boxplus 1/t}}$ directly, 
but not easy except for special cases, see the example below. 
This is why, we try to understand the behavior of support of $\mu^{\boxplus 1/t}$ by obtaining good bounds.

Firstly, let us recall the following result by Belinschi and Bercovic \cite{BB2004}.
\begin{theorem}
  Given $\mu\in M_{\mathbb{R}}$, for any $T>1$, there exists
a measure $\mu^{\boxplus T}$ such that 
\[
\varphi_{\mu^{\boxplus T}}(z)=T\varphi_{\mu}(z). 
\]
Moreover, there exists an injective analytic map  $\omega_T:\mathbb{C}^+\rightarrow\mathbb{C}^+$
such that 
\begin{enumerate}
  \item $\mathit{F}_{\mu^{\boxplus T}}(z)=\mathit{F}_{\mu}(\omega_T(z))$;
\item $\omega_T(z)=\frac{1}{T}z+\frac{T-1}{T}\mathit{F}_{\mu^{\boxplus T}}(z)$, and hence 
 the function $\omega_T$ is the right inverse of the function 
      \[ H_T(z)=Tz+(1-T)\mathit{F}_{\mu}(z)
           =z+(T-1)(z-\mathit{F}_{\mu}(z)), \]
    for $z\in\mathbb{C}^+$. In other words, $H_T(\omega_T(z))=z$; 
\item  The function $\omega_T$ has a continuous, one-to-one 
extension to the set $\overline{\mathbb{C}^+}$ (see also \cite{BB2005}).
\end{enumerate}
\end{theorem}

Secondly, let $\Omega_T=\omega_T(\mathbb{C}^+)$, so that
\[
\Omega_T = \{ z \in \mathbb C^+: \Im H_T (z)>0 \} 
\]
because $\omega_T(H_T(z)) =z$ on $\Omega_T$ has the analytic extension to the right hand side. 
Importantly, we can analyze $\Omega_T$ by using the funciton $H_T$,
in fact;
\begin{lemma}\label{lemma:Omega} 
For a self-adjoint operator $a\in\mathcal{A}$, let $\mu=\mu_a$. Then, 
there exists a continuous function
$f_T:\mathbb{R}\rightarrow \mathbb{R}^+\cup \{0\} $ such that
\[
 \partial\Omega_T=\{x+iy: y=f_T(x) \}. 
\]
\end{lemma} 
\begin{proof} 

First, we  write $\mathit{F}_{\mu}$ as
\[ 
  \mathit{F}_{\mu}(z)=-\tau(a)+z+\int_{\mathbb{R}}\frac{1}{\xi-z}d\rho(\xi)
\]
where 
$\rho$ is a positive measure satisfying $\rho(\mathbb{R})=\sigma(a)^2=\tau(a^2)-\tau(a)^2$,
according to \cite[Proposition 2.2]{Maassen}. 
Then, we have 
\begin{equation}
  H_T(z)=z+ (T-1)\left(\tau(a)+ \int_{\mathbb{R}}\frac{1}{z-\xi}d\rho(\xi) \right)
\end{equation}
and
\[
 \Im H_T(x+iy)=y\left(1-\int_{\mathbb{R}}\frac{T-1}{(\xi-x)^2+y^2}d\rho(\xi)\right).
\]
Next, notice that we always can find large enough $y$ such that $ \Im H_T(x+iy) >0$.  
Finally, our claim follows from the fact that
the function \[y\rightarrow \int_{\mathbb{R}}\frac{1}{(\xi-x)^2+y^2}d\rho(\xi)\]
is a positive and decreasing function of $y$.
See \cite{Biane1997} and \cite{Huang} for more details.
\end{proof} 
From the above proof we get the following corollary:
\begin{corollary} \label{corollary:Omega} 
In the same settings of Lemma \ref{lemma:Omega}
\begin{enumerate}
\item $f_T(x) >0$ if and only if 
\[
 \frac{1}{T-1}< \int_{\mathbb{R}}\frac{1}{(\xi-x)^2}\rho(d\xi)
\]
\item In this case, $f_T(x)$ is the unique solution satisfying
\begin{equation}
    \frac{1}{T-1}=\int_{\mathbb{R}}\frac{1}{(\xi-x)^2+f_T(x)^2}d\rho(\xi),
\end{equation}
and 
\[
H_T(x+if_T(x)) \in \mathbb R.
\]
\end{enumerate} 
\end{corollary}

Thirdly, having Lemma \ref{lemma:Omega} and Corollary \ref{corollary:Omega} in our minds, we define two sets $A_T, B_T \subseteq \mathbb R$:
\begin{align}
  A_T &=\{H_T(x+if_T(x)): x\in B_T\}; \\
 B_T&=\{x:f_T(x)>0, x\in\mathbb{R} \}.\label{eq:02}
\end{align}
Here, it is easy to see that $A_T=\{x: \Im \omega_T(x)>0\}=\{x:\Im \mathit{F}_{\mu^{\boxplus T}}(x)>0\}$. 
Then, we have the following statement:
\begin{lemma} \label{lemma:inclusion}
If $B_T \subseteq [x_1,x_2]$ then $A_T \subseteq [H_T(x_1+if_T(x_1)),H_T(x_2+if_T(x_2))]$. 
\end{lemma}
This lemma can be proven from the fact that
the function 
\[
 x\rightarrow H_T(x+if_T(x))
\]
is increasing on $\mathbb{R}$.
For later use, we quote another result from the proof of \cite[Proposition 3]{Biane1997}:
\begin{proposition}
The support of $\rho$, $supp$$(\rho)\subset B_T$.
In particular, the function $H_T$ is defined as an analytic and increasing function
on $\mathbb{R}\backslash B_T$.
\end{proposition}

Fourthly,  we have the following statement:
\begin{lemma}\label{lemma:2.5}
Given any Hermitian operator $a\in\mathcal{A}$ such that $L^- \leq  a\leq L^+$ with $L^\pm \in \mathbb R$, we have for $T>1$
\[ 
A_T \subset [x_1(T),x_2(T)],
\]
where 
\begin{align}
x_1(T)&=L^- -2\sigma(a)\sqrt{T-1}+(T-1)\tau(a)\label{eq:009}\\
x_2(T)&=L^+ + 2\sigma(a)\sqrt{T-1}+(T-1)\tau(a).\label{eq:010}
\end{align}
\end{lemma}
\begin{proof}
By the choice of $\rho$, we have
supp$(\rho)\subset [L^-,L]$.
Recall that $\rho(\mathbb{R})=\sigma(a)^2$.
From (\ref{eq:02}), it is easy to see that 
\[
  B_T\subset 
   [L^- -\sigma(a)\sqrt{T-1},L^+ + \sigma(a)\sqrt{T-1}].
\]
We estimate
\begin{equation}\nonumber
 \begin{split}
 &H_T(L^+ + \sigma(a)\sqrt{T-1})\\
 &\leq L^+ + \sigma(a)\sqrt{T-1}+(T-1)\left[\tau(a)+ \frac{\sigma(a)}{\sqrt{T-1}} \right]\\
 &=L^+ + 2\sigma(a)\sqrt{T-1}+(T-1)\tau(a);
 \end{split}
\end{equation}
and
\[
\begin{split}
  &H_T(L^--\sigma(a)\sqrt{T-1})\\
  &\geq L^- -\sigma(a)\sqrt{T-1}+(T-1)\left[\tau(a)-\frac{\sigma(a)}{\sqrt{T-1}} \right]\\
   &=L^--2\sigma(a)\sqrt{T-1}+(T-1)\tau(a). 
\end{split}
\]
Lemma \ref{lemma:inclusion} completes the proof. 
\end{proof}

To take care of atomic part of $\mu^{\boxplus T}$, we quote the following result from \cite{BB2004}.
\begin{theorem}\label{thm:2.6}
Given $T>1$, a number $\alpha$ is an atoms of $\mu^{\boxplus T}$
if and only if $\alpha/T$ is an atom of $\mu$ such that 
$\mu(\{\alpha/T \})>1-\frac{1}{T}$. In this case,
\[
   \mu^{\boxplus T}(\{\alpha \})=T\mu\left(\left\{\frac{\alpha}{T} \right\}\right)-(T-1).
\]
\end{theorem}
Finally, note that the Cauchy transform of $\mu=\mu_a$ is
\[
G_{\mu} (z)= \sum_i \frac 1 {z-\xi_i}
\]
where $\xi_i$ are the eigenvalues of Hermitian matrix $a$.
Then, by using Proposition \ref{prop:SP}, Lemma \ref{lemma:2.5} and Theorem \ref{thm:2.6},
we obtain an upper bound for $(t)$-norm:

\begin{theorem}\label{thm:upper-bound}
Take a Hermitian matrix $a\in M_k$ such that its eigenspaces have dimensions $D_1,\ldots, D_m$ in descending order,
with corresponding eigenvalues $\xi_1,\ldots,\xi_m$. 
Let $\sigma(a)^2=\tau(a^2)-(\tau(a))^2$, and $t\in (0,1]$. 
\begin{enumerate}
\item  If $D_1 \leq k(1-t)$, then the measure $\mu_a^{\boxplus 1/t}$ has no atomic part, and 
\[
||a||_{(t)}\leq \max \left\{ \left| t L^\pm \pm 2 \sigma(a) \sqrt{t(1-t)} + (1-t) \tau(a) \right| \right\},
\]
for $L^- \leq a \leq L^+$ as in Lemma \ref{lemma:2.5}.
\item 
The above bound holds if all the eigenvalues of $a$ are different and $t \leq 1-\frac 1k$. 
 If, in addition, $\tau(a) =0$, 
\[
 ||a||_{(t)}\leq t L + 2 \sigma(a) \sqrt{t(1-t)} 
\]
for $L = \| a\|$.
\item If $D_1 > k(1-t)$, then  the measure $\mu_a^{\boxplus 1/t}$ has atomic parts, and
\[
||a||_{(t)}\leq \max \left\{\xi_{i_o}, \left| t L^\pm \pm 2 \sigma(a) \sqrt{t(1-t)} + (1-t) \tau(a) \right| \right\},
\]
where 
\[
\xi_{i_0} = \max \{ \xi_i : D_i > k(1-t) \}
\]
\end{enumerate}
\end{theorem}

\begin{remark}
We have 
$||a-\tau(a)||= \max\{\tau(a), ||a||-\tau(a) \}$
for $a\geq 0$. Therefore, 
we see that the upper bound in Theorem \ref{thm:upper-bound} is better than
the one given by (\ref{eq:006}).
\end{remark}

\begin{example}
It is known in \cite{BB2004} that, if $\mu=\mu_a=\frac{1}{2}(\delta_{-1}+\delta_1)$,
then
\begin{enumerate}
  \item The measure $\mu^{\boxplus T}$ has atom at $T$ if and only if $1<T<2$;
  \item The absolutely continuous part of $\mu^{\boxplus T}$ is supported on $[-2\sqrt{T-1},2\sqrt{T-1}]$.
\end{enumerate}
When $\mu_a$ is such, then 
\[
 ||a||_{(t)}=\left\{
  \begin{array}{lr}
    1 & : \frac{1}{2}\leq t\leq 1\\
    2\sqrt{t(1-t)} & :0< t\leq \frac{1}{2}.
  \end{array}
\right.
\]
And our estimation reads
\[
 ||a||_{(t)}\leq 2\sqrt{t(1-t)}+t,
\]
for $t\in (0,1].$
\end{example}

We now turn to prove a lower bound by using similar estimations.
Let $x_3=\sup\{x:x\in A_T \}$ and $x_4=\sup\{x:x\in B_T \}$.
The next result, taken from \cite[Theorem 4.6]{BB2005} will be a key for us to obtain a lower bound.
\begin{proposition}
For any $z_1,z_2\in\overline{\Omega_T}$, we have
\[
 |H_T(z_1)-H_T(z_2)|\leq 2 |z_1-z_2|.
\]
\end{proposition}

\begin{proposition}
We have
\[
 x_3\geq -L+ \left(1+\epsilon(a,t)\right)\sigma(a)\sqrt{T-1},
     +(T-1)\tau(a).
\]
where \[\epsilon(a,t)=\frac{\sigma(a)\sqrt{T-1}}{L+\sigma(a)\sqrt{T-1}}\]
for $T>1$. 
\end{proposition}
\begin{proof}
We first note that the function $H_T$ is analytic on $\mathbb{R}\backslash \overline{B_T}$
and $H_T'(z)>0$ for $z\in\mathbb{R}\backslash \overline{B_T}$.
We concentrate on finding a lower bound
of $x_4$.

As it was pointed out in \cite{BCN1}, when $\mu=\mu_a$ the 
measure $\rho$ is a compactly supported purely atomic positive measure on $\mathbb{R}$.
We also know that $supp(\rho)\subset B_T$. 
Therefore the function $H_T$ is a rational function.
From the definition of $H_T$
and $B_T$,
it is not hard to see that $x_4$ is the largest in solution to the equation
\[ 
  H'_T(z)=0, z\in\mathbb{R}. 
\]
We have
\[
 H_T'(z)=0\Leftrightarrow \frac{1}{T-1}=\int_{\mathbb{R}}\frac{1}{(\xi-z)^2}\rho(d\xi),
\]
which implies that
\[
  x_4\geq \sqrt{(T-1)\rho(\mathbb{R})}=\sigma(a)\sqrt{T-1}.
\]

We know that $L+\sigma(a)\sqrt{T-1}>x_4$. 
We have
\[
\begin{split}
 x_3 = H_T(x_4),
\end{split}
\]
and 
\[
 \begin{split}
   2L&  \geq 2(L+\sigma(a)\sqrt{T-1}-x_4)\\
      &\geq H_T(L+\sigma(a)\sqrt{T-1})-x_3,
 \end{split}
\]
thanks to Proposition 5.8.
We have
\[
   \begin{split}
    &H_T(L+\sigma(a)\sqrt{T-1})\\
      &\geq  L+\sigma(a)\sqrt{T-1}+(T-1)\left( \tau(a)+\frac{\sigma(a)^2}{L+\sigma(a)\sqrt{T-1}}\right)\\
      &\geq L+\left(1+\epsilon(a,t)\right)\sigma(a)\sqrt{T-1}+(T-1)\tau(a),
 \end{split}
\]
where \[\epsilon(a,t)=\frac{\sigma(a)\sqrt{T-1}}{L+\sigma(a)\sqrt{T-1}},\] for for all $T>1$. 
Finally, we obtain
\[
  x_3\geq -L+ \left(1+\epsilon(a,t)\right)\sigma(a)\sqrt{T-1}
     +(T-1)\tau(a).\qedhere
\]
\end{proof}

\begin{theorem}\label{thm:lower-bound}
Given $0\leq a\in\mathcal{A}$. Assume that $||a||\leq L$, for $0<t\leq 1$, we have
\[
||a||_{(t)}\geq \tau(a)+\left(1+\epsilon(a,1/t)\right)\sigma(a)\sqrt{t(1-t)}-t(L+\tau(a)),
\]
where 
\[\epsilon(a,t)=\frac{\sigma(a)\sqrt{t-1}}{L+\sigma(a)\sqrt{t-1}}.\] 
\end{theorem}

\section{Violation of the MOE additivity}\label{sec:violation}

\subsection{Non additivity for the MOE: strategy}\label{sec:strategy}

As announced in the introduction, we revisit the proof that the minimum output entropy is not additive. 
Let us first review the strategy of proof, that relies on the use of sequences of random random channels. 

The construction of counterexamples is based on the \emph{Stinespring dilation theorem} \cite{Stinespring1955}, which asserts that any completely 
positive, trace preserving map $\Phi : M_d (\mathbb C) \to M_k(\mathbb C)$ can be realized as
\begin{equation}\label{eq:channel-from-isometry}
\Phi(X) = [\mathrm{id}_k \otimes \mathrm{Tr_n}](WXW^*), 
\end{equation}
where $M_d(\mathbb C)$ is the $d\times d$ complex matrix algebra, and
 $n$ is an integer (called the dimension of the environment) so that
\begin{equation}
W : \mathbb C^d \to \mathbb C^k \otimes \mathbb C^n
\end{equation}
is an isometry; $W^*W = I_d$. 
Here, we set $d=tkn$ for $t \in (0,1]$.
Conversely, any isometry $W$ gives rise to a quantum channel.
Importantly, if we take random isometry $W$ with respect to the Haar measure of $\mathcal U(kn)$,
$WW^* \in M_{kn}(\mathbb C)$ is a random projection of rank $d$.

In \eqref{eq:channel-from-isometry}, replacing $\mathrm{id}_k \otimes \mathrm{Tr_n}$ by $\mathrm{Tr_k} \otimes \mathrm{id}_n$
gives the complementary channel. Output states of those two channels share the same non-zero eigenvalues 
when the input matrix $X$ is a rank-one projection. 
In this paper, we look for minimum values of the entropy function, which is concave,
therefore  we can assume that $X$ is always a rank-one projection. 
Hence, in particular, we can take $k \leq n$ without loss of generality.

Given $\Phi$, we define its complex conjugate $\overline \Phi$ as 
\begin{equation}
\overline\Phi(X) = [\mathrm{id} \otimes \mathrm{Tr}](\overline WXW^t),
\end{equation}
where $\overline W$ is the entry-wise conjugate of $W$ and the superscript $t$ stands for transpose. 
The following crucial but simple lemma justifies the choice of a pair $(\Phi,\overline \Phi)$ in order to obtain violation.
It is due to Hayden and Winter \cite{HaydenWinter2008}. 
\begin{lemma} For any quantum channel defined in \eqref{eq:channel-from-isometry},
$$||\Phi\otimes \overline \Phi (B)||_{\infty}\geq t$$
where $B$ stands for a Bell state and $t=\frac d {kn}$. 
\end{lemma}
This lemma is used as follows: if the operator norm of an output is large, then its entropy is somewhat small, and therefore
we have an upper bound on the MOE of $\Phi\otimes \overline \Phi$. This upper bound will turn out
to be sufficient for appropriate choices of channels although a proof can be written in several lines.  
This is what we describe in the following sections.
We close the current section by showing the upper bound mentioned above:
\begin{theorem}\label{thm:product-bound}
For any quantum channel defined in \eqref{eq:channel-from-isometry} with $t \geq k^{-2}$
\[
H^{\min} (\Phi\otimes \overline \Phi )  \leq 2(1-t) \log(k)+h(t). 
\]
Here, $h(t)=-t\log t - (1-t)\log(1-t)$ is the binary entropy. 
\end{theorem} 
\begin{proof}
Since the size of the matrix  $\Phi\otimes \overline \Phi (B)$ is $k^2$, 
the highest entropy with the largest eigenvalue being $t$ is given when
the other $k^2-1$ eigenvalues are all $\frac{1-t}{k^2-1}$.
This gives an upper bound;
\[
H(\Phi\otimes \overline \Phi (B)) \leq -t \log t - (1-t) \left[\log (1-t) - \log (k^2-1)\right].
\]
\end{proof}

\subsection{Random model} \label{sec:model}

Now we fix a number $t\in (0,1)$ and $k$, and define a sequence of random quantum channels in $n\in\mathbb N$ as follows.
First, we define a random isometry:
\[
V_n: \mathbb C^{tkn} \to \mathbb C^k \otimes \mathbb C^n
\]
 by taking the first $\lfloor tkn \rfloor$ columns of unitary matrices in $\mathcal U(kn)$ with the Haar measure.
Then we define a random quantum channel 
\begin{equation}\label{eq:RQ}
\Phi_n (X) = \Tr_{\mathbb C^n} [V_n X V_n^*].
\end{equation}

The collection of eigenvalues of $\Phi_n(X)$ 
is denoted by $K_{k,n,t}$
and is a random subset of the
the simplex $\Delta_k$.

Now, we introduce the convex body $K_{k,t}\subset \Delta_k$ as follows:
\begin{equation}\label{eq:convex}
        K_{k,t}:=\{ \lambda \in\Delta_{k} \; |\; \forall a\in\Delta_{k} , \langle \lambda ,a \rangle \leq || a ||_{(t)}
        \},
\end{equation}
where $\langle\cdot , \cdot \rangle$ denotes the canonical scalar product in $\R^k$.
The behaviour of the limiting convex set was proven in \cite{BCN1} and is recalled below:
\begin{theorem}\label{thm:main-bcn1}
Almost surely, the following statements hold true:
\begin{itemize}
\item Let $\mathcal{O}$ be an open set in $\Delta_{k}$ containing $K_{k,t}$.
Then, for $n$ large enough, $K_{n,k,t}\subset \mathcal{O}$.
\item Let $\mathcal K$ be a compact set in the interior of $K_{k,t}$.
Then, for $n$ large enough, $\mathcal K \subset K_{n,k,t}$.
\end{itemize}
\end{theorem}
This result was re-investigated in \cite{CFN3} in $M_k (\mathbb C)$ with Euclidean norm,
which is used below.

\subsection{Concentration around the maximally mixed state} \label{sec:concentration}

When $k \ll n$, it was proved explicitly   
in \cite{BrandaoHorodecki2010,AubrunSzarekWerner2011,Fukuda2014} 
by using measure concentration and large deviation that
\[
\{\Phi_n (X) : X \in S_{d}\}
\]
concentrates around the maximally mixed state $\tilde I = I_k /k$,
bounding the following quantity:
\[
\max_{X \in S_{d}}\| \Phi_n (X) - \tilde I \|_2.
\]

Now, we do the same by using Theorem  \ref{thm:upper-bound}, which was proven via free probability. 
Indeed, Theorem \ref{thm:main-bcn1} and \eqref{eq:convex} imply that
with probability one for the sequence of random channels $\{\Phi_n\}_{n \in \mathbb N}$ we have
$\forall \epsilon >0, \exists n_\epsilon \in \mathbb N, \forall n \geq n_\epsilon, \forall X \in S_n$ 
\begin{equation}\label{eq:trace-trick1}
\Tr (\Phi_n (X) - \tilde I)^2 = \Tr \left[\Phi_n(X) (\Phi_n(X) - \tilde I)\right] \leq  ||\Phi_n (X )-\tilde I||_{(t)}\cdot(1+\epsilon).
\end{equation}
Moreover, ignoring contributions of atomic parts in Theorem \ref{thm:upper-bound} we have
\begin{align}\label{eq:trace-trick2}
 ||\Phi_n (X )-\tilde I||_{(t)} &\leq  t\left[||\Phi_n(X) - \tilde I ||+2\sigma (\Phi_n(X) - \tilde I)\sqrt{1/t-1}\right] \notag\\
 & \leq t\left[||\Phi_n(X) - \tilde I ||_2+2 \sqrt{1/k \cdot  ||\Phi_n(X) - \tilde I ||_2^2}   \sqrt{1/t-1}\right]  \notag\\
 & =t||\Phi_n (X )-\tilde I ||_2 \left[1+2\sqrt{\frac{1-t}{tk}}\right].
\end{align}
Here, notice that $\tau (\Phi(X) - \tilde I)=0$ and that the operator norm is bounded above by the non-normalized $L^2$ norm.
Also, remember that $\sigma(X)^2=k^{-1}\|X\|_2^2$ when $\Tr X =0$.

Hence by using \eqref{eq:trace-trick1} and \eqref{eq:trace-trick2} we get our desired result:
almost surely, for the sequence of random channels $\{\Phi_n\}_{n \in \mathbb N}$ we have
$\forall \epsilon >0, \exists n_\epsilon \in \mathbb N, \forall n \geq n_\epsilon, \forall X \in S_n$ 
\[
||\Phi_n (X )-\tilde I ||_2 \leq t\left[1+2\sqrt{\frac{1-t}{tk}}\right]\cdot( 1+ \epsilon).
\]
So far, we have not taken account of contributions of atoms to $(t)$-norm. 
However, 
we can assume without loss of generality that all the eigenvalues of $\Phi (X )-\tilde I $ are different
otherwise we can perturb them a little to avoid having multiplicity of eigenvalues. 
Hence, according to Theorem  \ref{thm:upper-bound}, the above bound holds for $t  \leq 1- \frac 1 k$.

Summarizing, we obtain the following theorem

\begin{theorem} \label{thm:bound2}
For $0< t \leq 1- \frac 1k$,
the following bounds hold true with probability $1$.
\[
\lim_{n \to \infty} \max_{X \in S_n} ||\Phi_n (X)-\tilde I ||_2 \leq t \left[1+2\sqrt{\frac{1-t}{tk}}\right].
\]
\end{theorem}

\subsection{Violation}\label{sec:violation2}
First, to prove violation of additivity for $H^{\min}$, we quote the following bound:
\[
\log k - H(X)  \leq k \cdot \Tr (X-\tilde I)^2 
\]
for $X \in S_k (\mathbb C)$.
This bound was used  in \cite{Hastings2009},
which can be proven by using concavity of $\log$.
Then, we apply the above bound to Theorem \ref{thm:bound2}
to show the following bounds hold true with probability $1$ 
for $t \leq 1- \frac 1k$. 
\begin{align}
\lim_{n \to \infty} H^{\min} \left(\Phi_n \right) \geq \log k -  t^2k \left[1+2\sqrt{\frac{1-t}{tk}}\right]^2.
\end{align}

Next, from Theorem \ref{thm:product-bound}, we have
\begin{align*}
H^{\min} (\Phi\otimes \overline \Phi )  &\leq 2 \log k - 2t \log k - t \log t - (1-t) \log (1-t)  \\
& \leq 2 \log k - 2t \log k - t \log t  +2t(1-t)
\end{align*}
for small enough $t \in (0,1)$.

Finally, let $t \sim k^{-r}$ with $1\leq r <2$, i.e., $\lim_{k\to\infty} \frac t {k^{-r}} \in (0, +\infty)$.
Then, for all $n \in \mathbb N$ and large $k\in\mathbb N$,
\[
H^{\min} (\Phi_n \otimes \overline \Phi_n ) -2 \log k \lesssim -  k^{-r} \log k,
\]
and  almost surely for  large $k\in\mathbb N$,
\[
 2 \lim_{n \to \infty} H^{min} (\Phi_n) - 2 \log k \gtrsim - k^{-r} . 
\]

Now, we have proven the following statement:
\begin{theorem}\label{thm:violation} 
Take the random quantum channels $\Phi_n$ as defined in \eqref{eq:RQ} with $t \sim k^{-r}$ and $1\leq r < 2$, 
then with probability one
\[
\limsup_{n \to \infty} H^{\min} (\Phi_n \otimes \overline \Phi_n )  
< 2 \lim_{n \to \infty} H^{min} (\Phi_n),
\]
for large enough $k$.
\end{theorem}
Note that $H^{min} (\Phi)=H^{min} (\overline \Phi)$ for any quantum channel $\Phi$. 
Theorem \ref{thm:violation} actually implies Theorem \ref{thm:violationF}. 
Indeed, take $t=\frac 1 k$ and the complementary channel of $\Phi$ yields the desired statement.

To conclude our paper, we discuss on  possible values of $k$ for the additivity violation. 
First, note that 
for $\lambda = (\lambda_1,\ldots,\lambda_k) \in K_{k,t}$ and $1\leq i \leq k$
\[
l_{k,t}:=1-\varphi\left(\frac{k-1}k,t\right) \leq \lambda_i \leq \varphi \left( \frac 1k, t \right) =:u_{k,t}
\]
where $\varphi(a,b)= a+b -2ab + 2\sqrt{ab(1-a)(1-b)}$, see \cite{CollinsNechita2011}.

Next, with these bounds we can control the error of Taylor expansion of $H(\cdot)$ around $(1/k,\ldots,1/k)$. 
In fact, for $\lambda \in K_{k,t}$ there exists $\gamma \in \Delta_k$ where $\gamma_i$ is between $\lambda_i$ and $1/k$ so that
\begin{align*}
H(\lambda) &= \log k - \frac k 2 \sum_{i=1}^k \left(\lambda_i - \frac 1k\right)^2 
+ \sum_{i=1}^k  \frac 1{6\gamma_i^2}   \left(\lambda_i - \frac 1k\right)^3 \\
& \geq \log k - \left[ \frac k 2 + \frac {u_{k,t}-1/k}{6 l_{k,t}^2} \right] \sum_{i=1}^k \left(\lambda_i - \frac 1k\right)^2 \\
&  \geq \log k - \left[ \frac k 2 + \frac {u_{k,t}-1/k}{6 l_{k,t}^2} \right] t^2 \left[1+2\sqrt{\frac{1-t}{tk}}\right]^2 =:f(k,t)
\end{align*}
where we used Theorem \ref{thm:bound2}.
Namely, we want the following function to be negative:
\begin{align*}
g(k,r) &= 2(1-t) \log k + h(t)  -2 f(k,t)
\end{align*}
to observe additivity violation.
Here, $t = 1/k^r$ with $r \in [1,2)$.

Finally, Figure \ref{fig:1} shows  contour lines of $g(k,t)$ drawn by \emph{Wolfram Mathematica}. 
In the negative region, additivity violation necessarily happens.   
Figure \ref{fig:2} indicates the minimum possible dimension for the violation based on our method;
$k=31114$ (with $t=k^{1.387}$), which gives $g(k,t) = - 6.71108 \times10^{-12}$.
This value of $k$ is smaller than $3.9 \times 10^4$ which was obtained in \cite{FKM2010} with $t=k^{-1}$ in a similar setting, but larger than $183$ obtained in \cite{BCN2013}, which is known to be the best estimate in this context.

\begin{figure}
\centering
\begin{minipage}{.5\textwidth}
  \centering
  \includegraphics[width=1\linewidth]{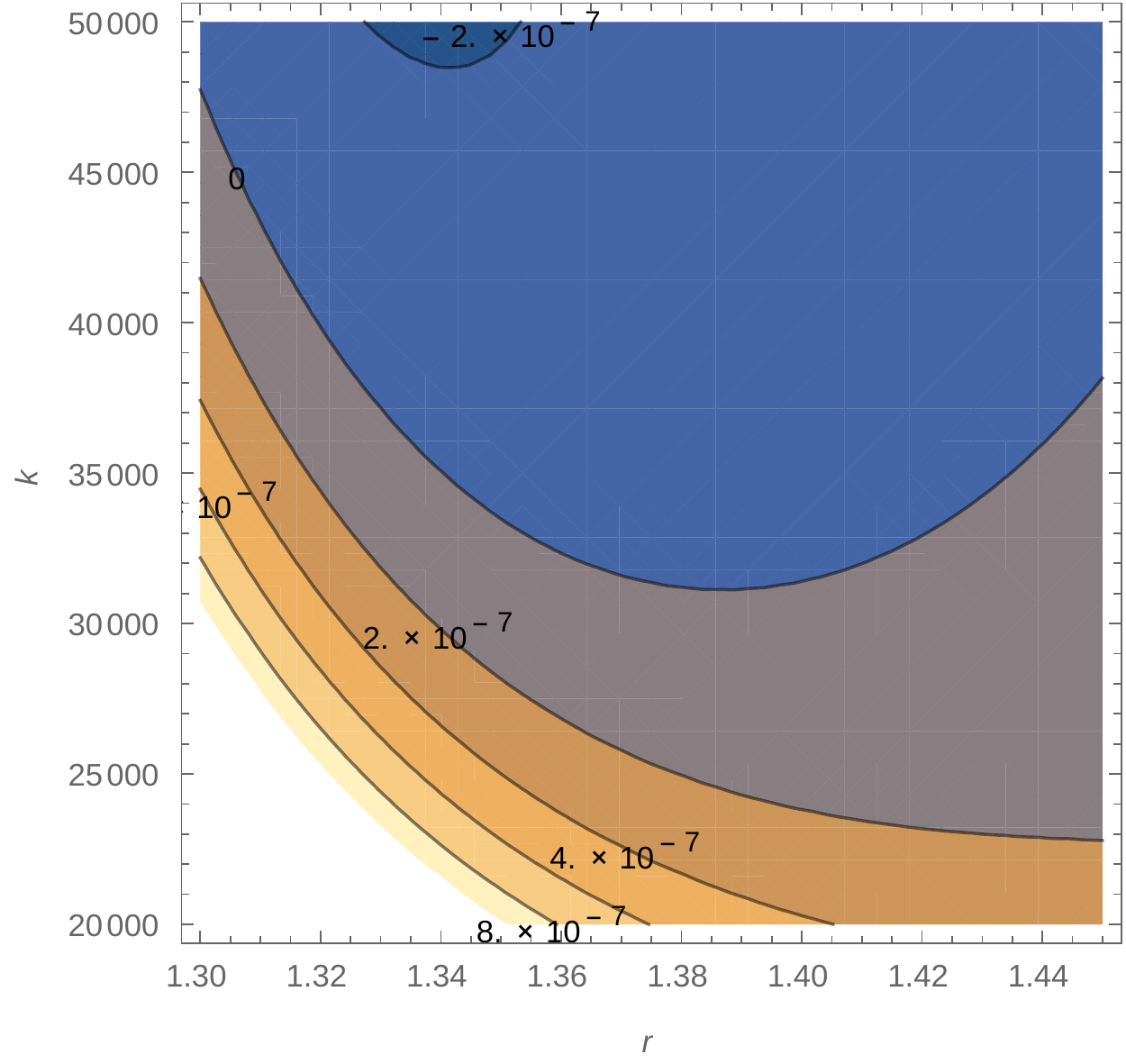}
  \captionof{figure}{Contour lines of $g(k,t)$.}
  \label{fig:1}
\end{minipage}%
\begin{minipage}{.5\textwidth}
  \centering
  \includegraphics[width=1\linewidth]{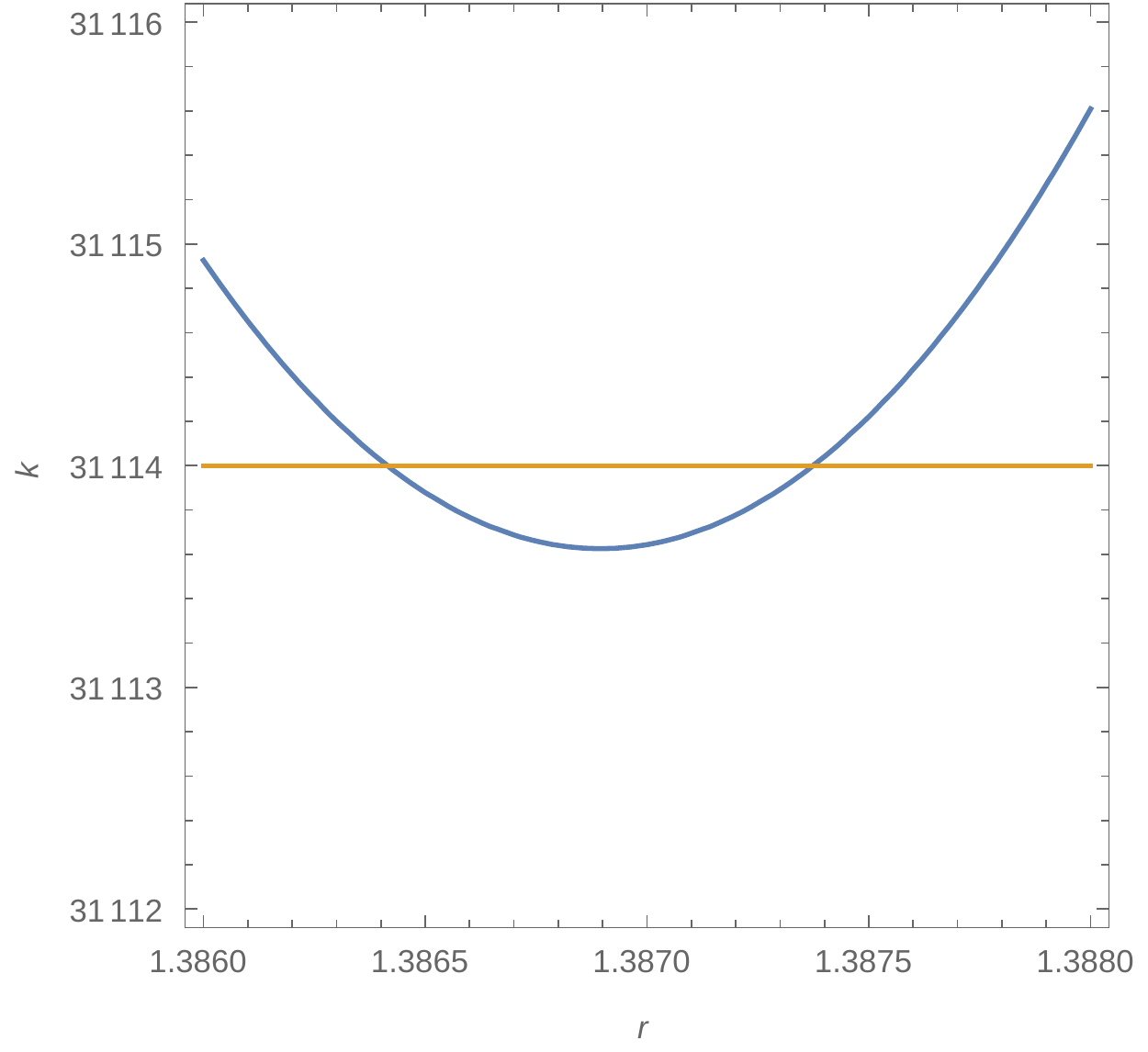}
  \captionof{figure}{$g(k,t)=0$ and $k=31114$.}
  \label{fig:2}
\end{minipage}
\end{figure}

\bibliographystyle{alpha}
\bibliography{ref_tnorm}

\end{document}